\documentclass[a4paper,12pt,reqno]{amsart}
\usepackage{amsfonts}
\usepackage{amsmath}
\usepackage{amssymb}
\usepackage{mathrsfs}
\usepackage{hyperref}
\renewcommand\eqref[1]{(\ref{#1})} 
\numberwithin{equation}{section}
\theoremstyle{plain}
\newtheorem{thm}{Theorem}[section]
\newtheorem{prop}[thm]{Proposition}
\newtheorem{cor}[thm]{Corollary}
\newtheorem{lem}[thm]{Lemma}

\theoremstyle{definition}

\renewcommand{\wp}{\mathfrak S}

\begin{document}

   \title[$L^{2}$-Caffarelli-Kohn-Nirenberg inequalities]
   {Anisotropic $L^{2}$-weighted Hardy and $L^{2}$-Caffarelli-Kohn-Nirenberg inequalities}

\author[M. Ruzhansky]{Michael Ruzhansky}
\address{
  Michael Ruzhansky:
  \endgraf
  Department of Mathematics
  \endgraf
  Imperial College London
  \endgraf
  180 Queen's Gate, London SW7 2AZ
  \endgraf
  United Kingdom
  \endgraf
  {\it E-mail address} {\rm m.ruzhansky@imperial.ac.uk}
  }
\author[D. Suragan]{Durvudkhan Suragan}
\address{
  Durvudkhan Suragan:
  \endgraf
  Institute of Mathematics and Mathematical Modelling
  \endgraf
  125 Pushkin str.
  \endgraf
  050010 Almaty
  \endgraf
  Kazakhstan
  \endgraf
  and
  \endgraf
  Department of Mathematics
  \endgraf
  Imperial College London
  \endgraf
  180 Queen's Gate, London SW7 2AZ
  \endgraf
  United Kingdom
  \endgraf
  {\it E-mail address} {\rm d.suragan@imperial.ac.uk}
  }

\thanks{The authors were supported in parts by the EPSRC
 grant EP/K039407/1 and by the Leverhulme Grant RPG-2014-02,
 as well as by the MESRK grant 5127/GF4. No new data was collected or generated during the course of research.}


     \begin{abstract}
       We establish sharp remainder terms of the $L^{2}$-Caffarelli-Kohn-Niren\-berg inequalities
     on homogeneous groups, yielding the inequalities with best constants. Our methods also give new sharp Caffarelli-Kohn-Nirenberg type inequalities in $\mathbb{R}^{n}$ with arbitrary quasi-norms. We also present explicit examples to illustrate our results for different weights and in abelian cases.
     \end{abstract}
     \maketitle

\section{Introduction}
Recall the following weighted Hardy-Sobolev type inequalities due to Caffarelli, Kohn and Nirenberg \cite{CKN-1984}:
For all $f\in C_{0}^{\infty}(\mathbb{R}^{n})$ we have the estimate
\begin{equation}\label{CKN1}
\left(\int_{\mathbb{R}^{n}}\|x\|^{-p\beta}|f|^{ p} dx\right)^{\frac{2}{p}}\leq C_{\alpha,\beta}\int_{\mathbb{R}^{n}}\|x\|^{-2\alpha}|\nabla f|^{2}dx,
\end{equation}
where for $n\geq3$:
$$-\infty<\alpha<\frac{n-2}{2},\; \alpha\leq \beta\leq\alpha+1,\;{\rm and} \; p=\frac{2n}{n-2+2(\beta-\alpha)},$$
and for
$n=2$:
$$-\infty<\alpha<0,\; \alpha<\beta\leq\alpha+1,\;{\rm and}\; p=\frac{2}{\beta-\alpha},$$
and where $\|x\|=\sqrt{x^{2}_{1}+\ldots+x^{2}_{n}}$. Inequality \eqref{CKN1}
is now knows as the Caffarelli-Kohn-Nirenberg inequality.

Nowadays there exists an extensive literature on Caffarelli-Kohn-Nirenberg type inequalities and their applications. We refer to \cite{CW-2001}, \cite{WW-2003}, \cite{HK12} and a recent paper
\cite{CH16} for further discussions and references on this subject.
We also note that the analysis of the remainder terms in different inequalities has a long history, initiated by Brezis and Nirenberg in \cite[Corollary 1.1 and Remark 1.4]{Brez3} and then Brezis and Lieb \cite{Brez1} for Sobolev inequalities, Brezis and V\'azquez in \cite[Section 4]{Brez4} for Hardy inequalities, see also \cite{Brez2}, with many subsequent works in this subject, see e.g. \cite{ACR, WW-2003} and many others, and a more recent literature review in \cite{GM}. Here we will be interested in also obtaining some formulae for the remainder terms in the appearing inequalities.

In this note we are interested in the $L^{2}$ case ($p=2$) of \eqref{CKN1}, that is,
for all $f\in C_{0}^{\infty}(\mathbb{R}^{n}),$
\begin{equation}
\int_{\mathbb{R}^{n}}\|x\|^{-2(\alpha+1)}|f|^{2}dx\leq \widetilde{C}_{\alpha}\int_{\mathbb{R}^{n}}\|x\|^{-2\alpha}|\nabla f|^{2}dx,
\end{equation}
with any $n\geq2$ and $-\infty<\alpha<\frac{n-2}{2}$, which in turn can be presented for any $f\in C_{0}^{\infty}(\mathbb{R}^{n}\backslash\{0\})$ as
\begin{equation}\label{CKNinRn}
\left\|\frac{1}{\|x\|^{\alpha+1}}|f|\right\|_{L^{2}(\mathbb{R}^{n})}\leq C_{\alpha}\left\|\frac{1}{\|x\|^{\alpha}}|\nabla f|\right\|_{L^{2}(\mathbb{R}^{n})},
\end{equation}
for all $\alpha\in\mathbb{R}$.

\smallskip
The main goal of this note is to show analogues of \eqref{CKNinRn} on homogeneous (Lie) groups. We use some techniques from our recent preprint \cite{Ruzhansky-Suragan:identities}.  Although there is a certain overlap between these settings here we aim to explain that the obtained homogeneous group results are not only analogues of the known Euclidean results, but also they give new inequalities even in Abelian cases with arbitrary quasi-norms. We also shall note that our main result (see Theorem \ref{aHardy}) is an anisotropic generalisation of the classical $L^{2}$-weighted Hardy and $L^{2}$-Caffarelli-Kohn-Nirenberg inequalities, but it is not what should be called a `genuine' subelliptic version of the classical inequalities since we will not use a horizontal gradient. 
Indeed, since we do not ask for the group to be stratified, there may be neither sub-Laplacian in this
generality nor any `horizontal' gradients. We refer to \cite{Ruzhansky-Suragan:horizontal} for horizontal versions of these inequalities.

\smallskip

To the best of our knowledge, there is no example of a nilpotent Lie group that does not allow for any compatible family of dilations if the (topological) dimension of a nilpotent Lie group is less than nine. At the same time, it is known that the class of the homogeneous groups  is one of the most general subclasses
of the class of nilpotent Lie groups. However, it is also known that these classes are not equal since an example of a nine-dimensional nilpotent Lie group, which does not allow for any compatible family of dilations, was constructed by Dyer  \cite{Dyer-1970}. Most popular special cases of homogeneous groups 
are the isotropic and anisotropic Abelian groups $(\mathbb{R}^{n}; +)$, H-type groups, stratified groups, graded Lie groups and so on.

Before presenting our main results we discuss some necessary basic concepts of the homogeneous groups. 
We refer to the book \cite{FS-Hardy} by Folland and Stein as well as to the recent monograph \cite{FR} by Fischer and the first named author for further discussions in this direction.

A Lie group (on $\mathbb{R}^{n}$) $\mathbb{G}$ with the dilation
$$D_{\lambda}(x):=(\lambda^{\nu_{1}}x_{1},\ldots,\lambda^{\nu_{n}}x_{n}),\; \nu_{1},\ldots, \nu_{n}>0,\; D_{\lambda}:\mathbb{R}^{n}\rightarrow\mathbb{R}^{n},$$
which an automorphism of the group $\mathbb{G}$ for each $\lambda>0,$
is called a {\em homogeneous (Lie) group}. Throughout this note instead of $D_{\lambda}(x)$ we will use the simpler and shorter  notation $\lambda x$.

A continuous non-negative function
$$\mathbb{G}\ni x\mapsto |x|\in [0,\infty),$$
is called a {\em homogeneous quasi-norm} on a homogeneous group $\mathbb G$
if it has the properties:

\begin{itemize}
\item   $|x^{-1}| = |x|$ for all $x\in \mathbb{G}$,
\item  $|D_{\lambda}(x)|=\lambda |x|$ for all
$x\in \mathbb{G}$ and $\lambda >0$,
\item  $|x|= 0$ if and only if $x=0$.
\end{itemize}

Note that a family of linear mappings of the form
$$D_{\lambda}={\rm Exp}(A \,{\rm ln}\lambda)=\sum_{k=0}^{\infty}
\frac{1}{k!}({\rm ln}(\lambda) A)^{k},$$
where $A$ is a diagonalisable linear operator on the Lie algebra $\mathfrak{g}$ (of $\mathbb{G}$)
with eigenvalues $\nu_{k}$, is a morphism of the Lie algebra $\mathfrak{g}$ (of $\mathbb{G}$),
that is, a linear mapping
from $\mathfrak{g}$ to itself which respects the Lie bracket:
$$\forall X,Y\in \mathfrak{g},\, \lambda>0,\;
[D_{\lambda}X, D_{\lambda}Y]=D_{\lambda}[X,Y].$$

The Haar measure on $\mathbb{G}$ is denoted by $dx$. Recall that the Lebesque measure on $\mathbb R^{n}$ gives the Haar measure for $\mathbb{G}$ (see, for example, \cite[Proposition 1.6.6]{FR}).
If $|S|$ is the corresponding volume of a measurable set $S\subset \mathbb{G}$, then
we have the equalities
\begin{equation}
|D_{\lambda}(S)|=\lambda^{Q}|S| \quad {\rm and}\quad \int_{\mathbb{G}}f(D_{\lambda}(x))
dx=\lambda^{-Q}\int_{\mathbb{G}}f(x)dx.
\end{equation}
Here $Q$ is the homogeneous dimension (see \eqref{Qdef}) of the homogeneous group $\mathbb{G}$.
Moreover, there is a (unique)
positive Borel measure $\sigma$ on the
unit sphere
\begin{equation}\label{EQ:sphere}
\wp:=\{x\in \mathbb{G}:\,|x|=1\},
\end{equation}
such that for all $f\in L^{1}(\mathbb{G})$ we have the polar decomposition
\begin{equation}\label{EQ:polar}
\int_{\mathbb{G}}f(x)dx=\int_{0}^{\infty}
\int_{\wp}f(ry)r^{Q-1}d\sigma(y)dr.
\end{equation}
We refer to Folland and Stein \cite{FS-Hardy} for the proof, which can be also found in
\cite[Section 3.1.7]{FR}.
Let us fix a basis $\{X_{1},\ldots,X_{n}\}$ of the Lie algebra $\mathfrak{g}$ of the homogeneous group $\mathbb{G}$
such that
$$AX_{k}=\nu_{k}X_{k}$$
for each $1\leq k\leq n$, so that $A$ can be taken to be
\begin{equation}\label{EQ:mA0}
A={\rm diag}(\nu_{1},\ldots,\nu_{n}).
\end{equation}
Then each $X_{k}$ is homogeneous of degree $\nu_{k}$ and also
\begin{equation}\label{Qdef}
Q=\nu_{1}+\cdots+\nu_{n},
\end{equation}
which is called a homogeneous dimension of $\mathbb{G}$. Homogeneous groups are necessarily nilpotent and hence, in particular, the exponential mapping $\exp_{\mathbb G}:\mathfrak g\to\mathbb G$ is a global diffeomorphism.
The decomposition of ${\exp}_{\mathbb{G}}^{-1}(x)$ in the Lie algebra $\mathfrak g$ defines the vector
$$e(x)=(e_{1}(x),\ldots,e_{n}(x))$$
with
$${\exp}_{\mathbb{G}}^{-1}(x)=e(x)\cdot \nabla_{X}\equiv\sum_{j=1}^{n}e_{j}(x)X_{j},$$
where $\nabla_{X}=(X_{1},\ldots,X_{n})$.
On the other hand, it says that
$$x={\exp}_{\mathbb{G}}\left(e_{1}(x)X_{1}+\ldots+e_{n}(x)X_{n}\right).$$
By the homogeneity property this implies
$$rx\equiv D_{r}(x)={\exp}_{\mathbb{G}}\left(r^{\nu_{1}}e_{1}(x)X_{1}+\ldots
+r^{\nu_{n}}e_{n}(x)X_{n}\right),$$
that is,
$$
e(rx)=(r^{\nu_{1}}e_{1}(x),\ldots,r^{\nu_{n}}e_{n}(x)).
$$
Thus, since $r>0$ is arbitrary, without loss of generality taking $|x|=1$, we obtain
\begin{equation}
\frac{d}{d|rx|}f(rx) =  \frac{d}{dr}f({\exp}_{\mathbb{G}}
\left(r^{\nu_{1}}e_{1}(x)X_{1}+\ldots
+r^{\nu_{n}}e_{n}(x)X_{n}\right)).
\end{equation}
Denoting by
\begin{equation}\label{EQ:Euler}
\mathcal{R} :=\frac{d}{dr},
\end{equation}
for all $x\in \mathbb G$ this gives the equality
\begin{equation}\label{dfdr}
\frac{d}{d|x|}f(x)=\mathcal{R}f(x),
\end{equation}
for each homogeneous quasi-norm $|x|$ on a homogeneous group $\mathbb G.$
Thus, the operator $\mathcal{R}$ plays the role of the radial derivative on $\mathbb{G}$. It is not difficult to see that $\mathcal{R}$ is homogeneous of order -1.

In Section \ref{SEC:2} we give our main results and their proofs. In Section \ref{SEC:3} we discuss some special cases.

The authors would like to thank Haim Brezis for drawing our attention to some aspects of the subject.

\section{Main results and their proofs}
\label{SEC:2}

In this section we establish the $L^{2}$ version of the Caffarelli-Kohn-Nirenberg inequality on the homogeneous group
$\mathbb{G}$. This will be the consequence of the following exact remainder formula which we believe to be new already in the setting of the Euclidean space.
As we will observe, this equality readily implies an 
$L^{2}$-weighted version of the Hardy inequality which 
in this context coincides with the $L^{2}$-Caffarelli-Kohn-Nirenberg inequality. We also note that the results below hold {\em for arbitrary homogeneous quasi-norms on $\mathbb{G},$} this yielding new insights already in the Euclidean setting of $\mathbb{G}=\mathbb R^{n}$.

\begin{thm}\label{aHardy}
Let $\mathbb{G}$ be a homogeneous group
of homogeneous dimension $Q\geq 3$.
Then for every complex-valued function $f\in C^{\infty}_{0}(\mathbb{G}\backslash\{0\})$
and any homogeneous quasi-norm $|\cdot|$ on $\mathbb{G}$ we have
\begin{multline}\label{awH}
\left\|\frac{1}{|x|^{\alpha}}\mathcal{R} f\right\|^{2}_{L^{2}(\mathbb{G})}-\left(\frac{Q-2}{2}-\alpha\right)^{2}
\left\|\frac{f}{|x|^{\alpha+1}}\right\|^{2}_{L^{2}(\mathbb{G})}
\\=\left\|\frac{1}{|x|^{\alpha}}\mathcal{R} f+\frac{Q-2-2\alpha}{2|x|^{\alpha+1}}f
\right\|^{2}_{L^{2}(\mathbb{G})}
\end{multline}
for all $\alpha\in\mathbb{R}.$
\end{thm}

As a consequence, we obtain the $L^{2}$-Caffarelli-Kohn-Nirenberg type
inequality on the homogeneous group $\mathbb{G},$ with the sharp constant:
\begin{cor}\label{waHardy}
Let $\mathbb{G}$ be a homogeneous group
of homogeneous dimension $Q\geq 3$.
Then for all complex-valued functions $f\in C^{\infty}_{0}(\mathbb{G}\backslash\{0\})$ we have
\begin{equation}\label{awHardyeq-g}
\frac{|Q-2-2\alpha|}{2}\left\|\frac{f}{|x|^{\alpha+1}}\right\|_{L^{2}(\mathbb{G})}\leq
\left\|\frac{1}{|x|^{\alpha}}\mathcal{R} f\right\|_{L^{2}(\mathbb{G})},\quad \forall\alpha\in\mathbb R.
\end{equation}
If $\alpha\neq\frac{Q-2}{2}$, then constant in \eqref{awHardyeq-g} is sharp for any homogeneous quasi-norm $|\cdot|$ on $\mathbb{G}$, and the inequality \eqref{awHardyeq-g} is attained if and only if $f=0$.
\end{cor}

In the Abelian case ${\mathbb G}=(\mathbb R^{n},+)$, $n\geq 3$, we have
$Q=n$, $e(x)=x=(x_{1},\ldots,x_{n})$, so for any homogeneous quasi-norm $|\cdot|$ on $\mathbb R^{n}$, \eqref{awHardyeq-g} implies
a new inequality with the optimal constant:
\begin{equation}\label{Hardy1-r}
\frac{|n-2-2\alpha|}{2}\left\|\frac{f}{|x|^{\alpha+1}}\right\|_{L^{2}(\mathbb{R}^{n})}\leq
\left\|\frac{1}{|x|^{\alpha}}\frac{df}{d|x|}\right\|_{L^{2}(\mathbb{R}^{n})},
\end{equation}
for all $\alpha\in\mathbb{R}$. We observe that this inequality holds for any homogeneous quasi-norm on $\mathbb R^{n}$.
For the standard Euclidean distance $\|x\|=\sqrt{x^{2}_{1}+\ldots+x^{2}_{n}}$, by using Schwarz's inequality, this implies the $L^{2}$-Caffarelli-Kohn-Nirenberg inequality \cite{CKN-1984} for $\mathbb{G}\equiv\mathbb{R}^{n}$ with the optimal constant:
\begin{equation}\label{CKN}
\frac{|n-2-2\alpha|}{2}\left\|\frac{f}{\|x\|^{\alpha+1}}\right\|_{L^{2}(\mathbb{R}^{n})}\leq
\left\|\frac{1}{\|x\|^{\alpha}}\nabla f\right\|_{L^{2}(\mathbb{R}^{n})},\;\forall\alpha\in\mathbb{R},
\end{equation}
for all $f\in C_{0}^{\infty}(\mathbb{R}^{n}\backslash\{0\}).$
Here optimality $\frac{|n-2-2\alpha|}{2}$ of the constant in \eqref{CKN} was proved in \cite[Theorem 1.1. (ii)]{CW-2001} for $\alpha<\frac{n-2}{2}$ and $f\in H_{0}^{1}(\mathbb{R}^{n}\backslash\{0\}).$
Note in the case of the Hardy inequality
$\alpha=0$, inequality \eqref{Hardy1-r} with the Euclidean distance, i.e. $|x|=\|x\|$, can be also obtained as a direct consequence
of supersolution construction and
Agmon-Allegretto-Piepenbrink theory (see \cite[Prop.
4.2 and Lemma 5.1]{DFP-2014}).

\begin{proof}[Proof of Theorem \ref{aHardy}]

First let us prove the case when $\alpha=0$.
Namely, if $f\in C_{0}^{\infty}(\mathbb{G}\backslash\{0\})$ is a complex-valued function then, for
$Q\geq 3$, we have
\begin{equation}\label{EQ:expL2}
\left\|\mathcal{R} f\right\|^{2}_{L^{2}(\mathbb{G})}=
\left(\frac{Q-2}{2}\right)^{2}\left\|\frac{f}{|x|}\right\|^{2}_{L^{2}(\mathbb{G})}+
\left\|\mathcal{R} f+\frac{Q-2}{2}\frac{f}{|x|}\right\|^{2}_{L^{2}(\mathbb{G})}.
\end{equation}

Introducing polar coordinates $(r,y)=(|x|, \frac{x}{\mid x\mid})\in$ $(0,\infty)\times\wp$ on $\mathbb{G}$, where $\wp$ is the pseudo-sphere in \eqref{EQ:sphere}, and using the integration formula \eqref{EQ:polar}
one calculates
$$
\int_{\mathbb{G}}
\frac{|f(x)|^{2}}
{|x|^{2}}dx
=\int_{0}^{\infty}\int_{\wp}
\frac{|f(ry)|^{2}}
{r^{2}}r^{Q-1}d\sigma(y)dr
$$
$$
=-\frac{2}{Q-2}\int_{0}^{\infty} r^{Q-2} \,{\rm Re}\int_{\wp}
f(ry) \overline{\frac{df(ry)}{dr}}d\sigma(y)dr
$$
\begin{equation}\label{EQ:formula1}
=-\frac{2}{Q-2} {\rm Re}\int_{\mathbb{G}}
\frac{f(x)}{|x|}
\overline{\frac{d}{d|x|}f(x)}dx.
\end{equation}
Using notations
$$u:=u(x)=-\frac{2}{Q-2}\mathcal{R} f,$$
and
$$v:=v(x)=\frac{f}{|x|},$$
formula \eqref{EQ:formula1}
can be restated as
\begin{equation}\label{vu3}
\|v\|_{L^{2}(\mathbb{G})}^{2}={\rm Re}\int_{\mathbb{G}}v \overline{u} dx.
\end{equation}
Then we have
\begin{multline}
\|u\|_{L^{2}(\mathbb{G})}^{2}-\|v\|_{L^{2}(\mathbb{G})}^{2}=
\|u\|_{L^{2}(\mathbb{G})}^{2}-\|v\|_{L^{2}(\mathbb{G})}^{2}\\+
2\int_{\mathbb{G}} (|v|^{2}-{\rm Re}\, v \overline{u}) dx
=
\int_{\mathbb{G}} (|u|^{2}+|v|^{2}-2{\rm Re}\, v \overline{u}) dx
\\=\int_{\mathbb{G}} |u-v|^{2} dx,
\end{multline}
which gives \eqref{EQ:expL2}.

Now we note the equality, for any $\alpha\in\mathbb{R}$,
\begin{equation}\label{EQ:eqE}
\frac{1}{|x|^{\alpha}}\mathcal{R} f=\mathcal{R} \frac{f}{|x|^{\alpha}}
+\alpha \frac{f}{|x|^{\alpha+1}}.
\end{equation}
Indeed, this follows from
$$
\mathcal{R} \frac{f}{|x|^{\alpha}}=\frac{1}{|x|^{\alpha}}\mathcal{R} f+f \mathcal{R}  \frac{1}{|x|^{\alpha}}
$$
and using \eqref{dfdr},
$$
\mathcal{R}  \frac{1}{|x|^{\alpha}}=\frac{d}{dr}\frac{1}{r^{\alpha}}=-\alpha\frac{1}{r^{\alpha+1}}=
-\alpha\frac{1}{|x|^{\alpha+1}},\quad r=|x|.
$$
Then we can write
\begin{multline}
\left\|\frac{1}{|x|^{\alpha}}\mathcal{R} f\right\|^{2}_{L^{2}(\mathbb{G})}
=\left\|\mathcal{R} \frac{f}{|x|^{\alpha}}
+\frac{\alpha f}{|x|^{\alpha+1}}
\right\|^{2}_{L^{2}(\mathbb{G})}
\\=\left\|\mathcal{R} \frac{f}{|x|^{\alpha}}
\right\|^{2}_{L^{2}(\mathbb{G})}+2\alpha{\rm Re}\int_{\mathbb{G}}
\mathcal{R} \left(\frac{f}{|x|^{\alpha}}\right)
\frac{\overline{f}}{|x|^{\alpha+1}}dx
+\left\|\frac{\alpha f}{|x|^{\alpha+1}}
\right\|^{2}_{L^{2}(\mathbb{G})}.
\end{multline}
By \eqref{EQ:expL2} with $f$ replaced by $\frac{f}{|x|^{\alpha}}$, we have
using \eqref{EQ:eqE} that
\begin{multline}
\left\|\mathcal{R} \frac{f}{|x|^{\alpha}}
\right\|^{2}_{L^{2}(\mathbb{G})}=\left(\frac{Q-2}{2}\right)^{2}\left\|\frac{f}{|x|^{1+\alpha}}
\right\|^{2}_{L^{2}(\mathbb{G})}
\\+\left\|\frac{1}{|x|^{\alpha}}\mathcal{R} f+\frac{Q-2-2\alpha}{2|x|^{\alpha+1}}f
\right\|^{2}_{L^{2}(\mathbb{G})}.
\end{multline}
Introducing polar coordinates $(r,y)=(|x|, \frac{x}{\mid x\mid})\in (0,\infty)\times\wp$ on $\mathbb{G}$ and using formula \eqref{EQ:polar} for polar coordinates, we calculate
\begin{align*}
2\alpha{\rm Re}\int_{\mathbb{G}}
\mathcal{R} \left(\frac{f}{|x|^{\alpha}}\right)
\frac{\overline{f}}{|x|^{\alpha+1}}dx
& =2 \alpha{\rm Re}\int_{0}^{\infty}r^{Q-2}\int_{\wp}
\frac{d}{dr}\left(\frac{f(ry)}
{r^{\alpha}}\right)\frac{\overline{f(ry)}}{r^{\alpha}}d\sigma(y)dr
\\
&=\alpha\int_{0}^{\infty}r^{Q-2}\int_{\wp}
\frac{d}{dr}\left(\frac{|f(ry)|^{2}}
{r^{2\alpha}}\right)d\sigma(y)dr
\\&=-\alpha(Q-2)\left\|\frac{f}{|x|^{\alpha+1}}
\right\|^{2}_{L^{2}(\mathbb{G})}.
\end{align*}
Summing up all above we obtain
\begin{multline}
\left\|\frac{1}{|x|^{\alpha}}\mathcal{R} f\right\|^{2}_{L^{2}(\mathbb{G})}=
\left(\frac{Q-2}{2}-\alpha\right)^{2}\left\|\frac{f}{|x|^{\alpha+1}}
\right\|^{2}_{L^{2}(\mathbb{G})}
\\+\left\|\frac{1}{|x|^{\alpha}}\mathcal{R} f+\frac{Q-2-2\alpha}{2|x|^{\alpha+1}}f
\right\|^{2}_{L^{2}(\mathbb{G})},
\end{multline}
yielding \eqref{awH}.
\end{proof}

Before we start proof of Corollary \ref{waHardy} let us record the version of the Euler operator on homogeneous groups.
\begin{lem}\label{L:Euler}
Define the Euler operator
\begin{equation}\label{EQ:def-Euler}
E:=|x| \mathcal{R}.
\end{equation}
If $f:\mathbb G\backslash \{0\}\to\mathbb R$ is continuously differentiable, then
$$
Ef=\nu f
 \; \textrm{ if and only if }\;
 f(D_{r} x)=r^{\nu} f(x)\;\; (\forall r>0, x\not=0).$$
\end{lem}
\begin{proof}[Proof of Lemma \ref{L:Euler}]
If $f$ is positively homogeneous of order $\nu$, i.e. if $f(rx)=r^{\nu}f(x)$ holds for
all $r>0$ and $x\not=0$, then applying \eqref{dfdr} to such $f$ we get
$$
Ef=\nu f(x).
$$
Conversely, let us fix $x\not=0$ and define $g(r):=f(rx)$.
Using \eqref{dfdr}, the equality $Ef(rx)=\nu f(rx)$ means that
$$
g'(r)=\frac{d}{dr}f(rx)=\frac{1}{r} Ef(rx)=\frac{\nu}{r} f(rx)=\frac{\nu}{r}g(r).
$$
Consequently, $g(r)=g(1) r^{\nu}$, i.e. $f(rx)=r^{\nu} f(x)$ and thus $f$ is positively homogeneous of order $\nu$.
\end{proof}
\begin{proof}[Proof of Corollary \ref{waHardy}]
Let us argue that the constant $\frac{|Q-2-2\alpha|}{2}$ is sharp and never attained unless $f=0$.
If the equality in \eqref{awHardyeq-g} is attained, it follows that the terms on the left hand side of
\eqref{awH} are zero. That is, it means that
\begin{equation}\label{EQ:aux1}
\frac{1}{|x|^{\alpha}}\mathcal{R} f+\frac{Q-2-2\alpha}{2|x|^{\alpha+1}}f=0
\end{equation}
and hence $Ef=-\frac{Q-2-2\alpha}{2}f$. In view of Lemma \ref{L:Euler} the function $f$ must be positively homogeneous of order $-\frac{Q}{2}+1+\alpha$, that is, $\frac{f}{|x|^{1+\alpha}}$ must be positively homogeneous of order $-\frac{Q}{2}$  which is impossible, so that the constant is not attained unless $f=0$. Therefore, the constant $\frac{|Q-2-2\alpha|}{2}$ in \eqref{awHardyeq-g} is sharp.
\end{proof}

\section{Cases $\alpha=-1,\,0,\,1$}
\label{SEC:3}
In this section we consider special cases of \eqref{awH} when $\alpha=-1,\,0,\,1$ to illustrate importance of this general equality in different settings.

\subsection{The case $\alpha=-1$.} We have the following relation for the Euler operator
\begin{prop} For any $f\in L^{2}(\mathbb{G})$ with $E f\in L^{2}(\mathbb{G})$ we have
\begin{equation}\label{ECrel}
\left\|E f\right\|^{2}_{L^{2}(\mathbb{G})}=\left(\frac{Q}{2}\right)^{2}\left\|f
\right\|^{2}_{L^{2}(\mathbb{G})}+\left\|E f+\frac{Q}{2}f
\right\|^{2}_{L^{2}(\mathbb{G})}.
\end{equation}
\end{prop}

\begin{proof}
Taking $\alpha=-1$, from \eqref{awH} we obtain
\eqref{ECrel} for any $f\in C_{0}^{\infty}(\mathbb{G}\backslash\{0\})$. Since $C_{0}^{\infty}(\mathbb{G}\backslash\{0\})$ is dense in $L^{2}(\mathbb{G})$, this implies that \eqref{ECrel} is also true on
$L^{2}(\mathbb{G})$ by density. The proof is complete.
\end{proof}

Simply by dropping the positive term in the right hand side,  \eqref{ECrel} implies
\begin{cor} For any $f\in L^{2}(\mathbb{G})$ with $Ef\in L^{2}(\mathbb{G})$
\begin{equation}
\left\|f
\right\|_{L^{2}(\mathbb{G})}\leq\frac{2}{Q}\left\|E f\right\|_{L^{2}(\mathbb{G})},
\end{equation}
with the best constant $\frac{2}{Q}$.
\end{cor}

\subsection{The case $\alpha=0$.}
In this case \eqref{awH} gives the equality
\begin{equation}\label{47-0}
\left\|\mathcal{R} f\right\|^{2}_{L^{2}(\mathbb{G})}=\left(\frac{Q-2}{2}\right)^{2}\left\|\frac{1}{|x|}f
\right\|^{2}_{L^{2}(\mathbb{G})}+\left\|\mathcal{R} f+\frac{Q-2}{2|x|}f
\right\|^{2}_{L^{2}(\mathbb{G})}.
\end{equation}
Now by dropping the nonnegative last term in \eqref{47-0} we immediately obtain a version of Hardy's inequality on $\mathbb{G}$ (see \cite{Ruzhansky-Suragan:Layers}- \cite{Ruzhansky-Suragan:uncertainty} for weighted $L^{p}$, critical, higher order cases and their applications in different settings):
\begin{equation}\label{47-1}
\left\|\frac{1}{|x|}f
\right\|_{L^{2}(\mathbb{G})}\leq\frac{2}{Q-2}\left\|\mathcal{R} f\right\|_{L^{2}(\mathbb{G})},
\end{equation}
again with $\frac{2}{Q-2}$ being the best constant.
Note that in comparison to stratified (Carnot) group versions, here the constant is best
for any homogeneous quasi-norm $|\cdot|$.

In the Abelian case ${\mathbb G}=(\mathbb R^{n},+)$, $n\geq 3$, we have
$Q=n$, $e(x)=x=(x_{1},\ldots,x_{n})$, so for any quasi-norm $|\cdot|$ on $\mathbb R^{n}$ it implies
the new inequality:
\begin{equation}\label{Hardy-r}
\left\|\frac{f}{|x|}\right\|_{L^{2}(\mathbb{R}^{n})}\leq \frac{2}{n-2}
\left\|\frac{df}{d|x|}\right\|_{L^{2}(\mathbb{R}^{n})},
\end{equation}
which in turn, by using Schwarz's inequality with the standard Euclidean distance $\|x\|=\sqrt{x^{2}_{1}+\ldots+x^{2}_{n}}$, implies the classical Hardy inequality for $\mathbb{G}\equiv\mathbb{R}^{n}$:
\begin{equation*}\label{Hardy}
\left\|\frac{f}{\|x\|}\right\|_{L^{2}(\mathbb{R}^{n})}\leq
\frac{2}{n-2}\left\|\nabla f\right\|_{L^{2}(\mathbb{R}^{n})},
\end{equation*}
for all $f\in C_{0}^{\infty}(\mathbb{R}^{n}\backslash\{0\}).$
When $|x|\equiv\|x\|$ the remainder terms for \eqref{Hardy-r} have been analysed by Ioku, Ishiwata and Ozawa
\cite{IIO:Lp-Hardy}, see also Machihara, Ozawa and Wadade \cite{MOW:Hardy-Hayashi} as well as \cite{IIO}.

The inequality \eqref{47-1} implies the following Heisenberg-Pauli-Weyl type  uncertainly principle on homogeneous groups (see e.g. \cite{Ricci15}, \cite{Ruzhansky-Suragan:Hardy} and \cite{Ruzhansky-Suragan:Layers} for versions of Abelian and stratified groups):

\begin{prop}\label{Luncertainty}
Let $\mathbb{G}$ be a homogeneous group of homogeneous dimension
 $Q\geq 3$.
Then for each $f\in C^{\infty}_{0}(\mathbb{G}\backslash\{0\})$ and any homogeneous quasi-norm $|\cdot|$ on $\mathbb{G}$ we have
\begin{equation}\label{UP1}
\left\|f\right\|^{2}_{L^{2}(\mathbb{G})}
\leq\frac{2}{Q-2}\left\|\mathcal{R} f\right\|_{L^{2}(\mathbb{G})}\left\||x| f\right\|_{L^{2}(\mathbb{G})}.
\end{equation}
\end{prop}
\begin{proof}
From the inequality \eqref{47-1} we get
$$
\left(\int_{\mathbb{G}}\left|\mathcal{R}f\right|^{2}dx\right)^{\frac{1}{2}}\left(\int_{\mathbb{G}}|x|^{2}
|f|^{2}dx\right)^{\frac{1}{2}}\geq$$
 $$\frac{Q-2}{2}\left(\int_{\mathbb{G}}
\frac{|f|^{2}}{|x|^{2}}\,dx\right)^{\frac{1}{2}}\left(\int_{\mathbb{G}}|x|^{2}
|f|^{2}dx\right)^{\frac{1}{2}}
\geq\frac{Q-2}{2}\int_{\mathbb{G}}
|f|^{2}dx,$$
where we have used the H\"older inequality in the last line.
This shows \eqref{UP1}.
\end{proof}

In the Abelian case ${\mathbb G}=(\mathbb R^{n},+)$, we have
$Q=n$, $e(x)=x$, so that \eqref{UP1} implies
the uncertainly principle with any quasi-norm $|x|$:
\begin{equation}\label{UPRn-r}
\left(\int_{\mathbb R^{n}}
 |u(x)|^{2} dx\right)^{2}
 \\ \leq\left(\frac{2}{n-2}\right)^{2}\int_{\mathbb R^{n}}\left|\frac{du(x)}{d|x|} \right|^{2}dx
\int_{\mathbb R^{n}} |x|^{2} |u(x)|^{2}dx,
\end{equation}
which in turn implies
the classical
uncertainty principle for $\mathbb{G}\equiv\mathbb R^{n}$ with the standard Euclidean distance $\|x\|$:
\begin{equation}
\left(\int_{\mathbb R^{n}}
 |u(x)|^{2} dx\right)^{2}
 \\ \leq\left(\frac{2}{n-2}\right)^{2}\int_{\mathbb R^{n}}|\nabla u(x)|^{2}dx
\int_{\mathbb R^{n}} \|x\|^{2} |u(x)|^{2}dx,
\end{equation}
which is the Heisenberg-Pauli-Weyl uncertainly principle on $\mathbb R^{n}$.

\subsection{The case $\alpha=1$.}
If $\alpha=1$, \eqref{awH} gives the equality
\begin{equation}\label{47-11}
\left\|\frac{1}{|x|}\mathcal{R} f\right\|^{2}_{L^{2}(\mathbb{G})}=\left(\frac{Q-4}{2}\right)^{2}\left\|\frac{f}{|x|^{2}}
\right\|^{2}_{L^{2}(\mathbb{G})}
\\+\left\|\frac{1}{|x|}\mathcal{R} f+\frac{Q-4}{2|x|^{2}}f
\right\|^{2}_{L^{2}(\mathbb{G})}.
\end{equation}
Then \eqref{47-11}
implies the estimate
\begin{equation}\label{awHardyeq}
\left\|\frac{f}{|x|^{2}}
\right\|_{L^{2}(\mathbb{G})}\leq
\frac{2}{Q-4}\left\|\frac{1}{|x|}\mathcal{R} f\right\|_{L^{2}(\mathbb{G})}, \quad Q\geq 5,
\end{equation}
again with $\frac{2}{Q-4}$ being the best constant.

In the Abelian case ${\mathbb G}=(\mathbb R^{n},+)$, $n\geq 5$, we have
$Q=n$, $e(x)=x=(x_{1},\ldots,x_{n})$, so for any homogeneous quasi-norm $|\cdot|$ on $\mathbb R^{n}$ it implies
the new inequality:
\begin{equation}
\left\|\frac{f}{|x|^{2}}\right\|_{L^{2}(\mathbb{R}^{n})}\leq \frac{2}{n-4}
\left\|\frac{1}{|x|} \frac{df}{d|x|}\right\|_{L^{2}(\mathbb{R}^{n})},
\end{equation}
which in turn, again by using Schwarz's inequality with the standard Euclidean distance $\|x\|=\sqrt{x^{2}_{1}+\ldots+x^{2}_{n}}$, implies the weighted Hardy inequality for $\mathbb{G}\equiv\mathbb{R}^{n}$:
\begin{equation*}
\left\|\frac{f}{\|x\|^{2}}\right\|_{L^{2}(\mathbb{R}^{n})}\leq
\frac{2}{n-4}\left\|\frac{1}{\|x\|}\nabla f\right\|_{L^{2}(\mathbb{R}^{n})},
\end{equation*}
for all $f\in C_{0}^{\infty}(\mathbb{R}^{n}\backslash\{0\}).$


\begin{thebibliography}{H8}

\bibitem{ACR}
Adimurthi, N. Chaudhuri and N. Ramaswamy.
An improved Hardy Sobolev inequality and its applications. 
{\em Proc. Amer. Math. Soc.} 130: 489--505, 2002.

\bibitem{Brez1}
H. Brezis and E. Lieb.
\newblock Sobolev inequalities with remainder terms.
\newblock {\em J. Funct. Anal.}, 62:73--86, 1985.

\bibitem{Brez2}
H. Brezis and M. Marcus.
\newblock Hardy's inequalities revisited.
\newblock {\em Ann. Scuola
Norm. Sup. Pisa Cl. Sci.}, 25(4):217--237, 1997.

\bibitem{Brez3}
H. Br\'ezis and L. Nirenberg.
Positive solutions of nonlinear elliptic equations involving critical Sobolev exponents. 
{\em Comm. Pure Appl. Math.}, 36(4): 437--477, 1983. 

\bibitem{Brez4}
H. Brezis and J. V\'azquez. 
Blow-up solutions of some nonlinear elliptic problems. 
{\em Rev. Mat. Univ. Complut. Madrid},  10(2): 443--469, 1997.

\bibitem{CKN-1984}
L.~Caffarelli, R.~Kohn, and L.~Nirenberg.
\newblock First order interpolation inequalities with weights.
\newblock {\em Compositio Mathematica}, 53:259--275, 1984.

\bibitem{CW-2001}
F. Catrina and Z.~Q. Wang.
\newblock On the Caffarelli-Kohn-Nirenberg inequalities: sharp constants, existence
(and non existence), and symmetry of extremals functions.
\newblock {\em Comm. Pure Appl. Math.}, 54:229--258, 2001.

\bibitem{Ricci15}
P.~Ciatti, M.~G.~Cowling, and F.~Ricci.
\newblock Hardy and uncertainty inequalities
on stratified Lie groups.
\newblock {\em Adv. Math.}, 227:365--387, 2015.


\bibitem{CH16}
N.~Chiba and T.~Horiuchi.
\newblock Radial symmetry and
its breaking in the Caffarelli-Kohn-Nirenberg type
inequalities for $p=1$.
\newblock {\em Proc. Japan Acad. Ser.
A Math. Sci.}, 92(4):51--55, 2016.

\bibitem{Dyer-1970}
J.~L. Dyer.
\newblock A nilpotent {L}ie algebra with nilpotent automorphism group.
\newblock {\em Bull. Amer. Math. Soc.}, 76:52--56, 1970.

\bibitem{DFP-2014}
B. Devyver, M. Fraas and Y. Pinchover. 
\newblock Optimal Hardy weight for second-order elliptic operator: An answer to a problem of Agmon.
\newblock {\em J. Funct. Anal.}, 266:4422--4489, 2014.

\bibitem{FR}
V.~Fischer and M.~Ruzhansky.
\newblock {\em Quantization on nilpotent Lie groups}, volume 314 of {\em
  Progress in Mathematics}.
\newblock Birkh\"auser, 2016.

\bibitem{FS-Hardy}
G.~B. Folland and E.~M. Stein.
\newblock {\em Hardy spaces on homogeneous groups}, volume~28 of {\em
  Mathematical Notes}.
\newblock Princeton University Press, Princeton, N.J.; University of Tokyo
  Press, Tokyo, 1982.

\bibitem{GM}
N. Ghoussoub and A. Moradifam.
Bessel pairs and optimal Hardy and Hardy-Rellich inequalities. 
{\em Math. Ann.}, 349(1):1--57, 2011.

\bibitem{HK12}
T.~Horiuchi and P.~Kumlin.
\newblock The Caffarelli-Kohn-Nirenberg type
inequalities involving critical and supercritical weights.
\newblock {\em Proc. Japan Acad. Ser.
A Math. Sci.}, 88(1):1--6, 2012.

\bibitem{IIO:Lp-Hardy}
N.~Ioku, M.~Ishiwata, and T.~Ozawa.
\newblock Sharp remainder terms of {H}ardy type inequalities in ${L}^p$.
\newblock {\em preprint}, 2015.

\bibitem{IIO}
N.~Ioku, M.~Ishiwata, and T.~Ozawa.
\newblock Sharp remainder of a critical {H}ardy inequality.
\newblock {\em Arch. Math. (Basel)}, 106(1):65--71, 2016.

\bibitem{MOW:Hardy-Hayashi}
S.~Machihara, T.~Ozawa, and H.~Wadade.
\newblock On the {H}ardy type inequalities.
\newblock {\em preprint}, 2015.

\bibitem{Ruzhansky-Suragan:Layers}
M.~Ruzhansky and D.~Suragan.
\newblock Layer potentials, {K}ac's problem, and refined {H}ardy inequality on
  homogeneous {C}arnot groups.
\newblock {\em arXiv:1512.02547}, 2015.

\bibitem{Ruzhansky-Suragan:critical}
M.~Ruzhansky and D.~Suragan.
\newblock Critical {H}ardy inequality on homogeneous groups.
\newblock {\em arXiv: 1602.04809}, 2016.

\bibitem{Ruzhansky-Suragan:Hardy}
M.~Ruzhansky and D.~Suragan.
\newblock Local {H}ardy and Rellich inequality for sums of squares of vector fields.
\newblock {\em Adv. Diff. Equations}, to appear, 2016.
arXiv:1605.06389

\bibitem{Ruzhansky-Suragan:identities}
M.~Ruzhansky and D.~Suragan.
\newblock Hardy and Rellich inequalities, identities, and sharp remainders on homogeneous groups.
\newblock {\em arXiv: 1603.06239}, 2016.

\bibitem{Ruzhansky-Suragan:uncertainty}
M.~Ruzhansky and D.~Suragan.
\newblock Uncertainty relations on nilpotent Lie groups.
\newblock {\em 	arXiv:1604.06702}, 2016.

\bibitem{Ruzhansky-Suragan:horizontal}
M.~Ruzhansky and D.~Suragan.
\newblock On horizontal Hardy, Rellich, Caffarelli-Kohn-Nirenberg and $p$-sub-Laplacian inequalities on stratified groups.
\newblock {\em J. Differential Equations}, to appear, 2016.
http://dx.doi.org/10.1016/j.jde.2016.10.028

\bibitem{WW-2003}
Z.~Q. Wang and M.~Willem.
\newblock Caffarelli-Kohn-Nirenberg inequalities with
remainder terms.
\newblock {\em J.  Funct. Anal.}, 203:550--568, 2003.


\end{thebibliography}
\end{document}